\documentclass[11pt]{article}
\usepackage{amsmath,amsthm,amssymb,amsfonts}
\usepackage{latexsym}
\usepackage{graphicx,psfrag,import}
\usepackage{fullpage}
\usepackage{framed}
\usepackage{verbatim}
\usepackage{color}
\usepackage{epsfig}
\usepackage{epstopdf}
\usepackage{hyperref}
\usepackage{geometry}
\usepackage{mathtools}
 \usepackage{nonfloat}
\usepackage{enumerate}
\usepackage{multicol}
\usepackage{a4wide}
\usepackage{booktabs}
\usepackage{enumitem}
\usepackage{lineno}
\usepackage{parcolumns}
\usepackage{thmtools}
\usepackage{xr}
\usepackage{mathrsfs}
\usepackage{subfig}
\usepackage{caption}
\usepackage{comment}
\usepackage{authblk}
\usepackage{setspace}
\usepackage{algorithm}

\usepackage{algorithmicx}

\usepackage{algpseudocode}

\usepackage{algcompatible}

\theoremstyle{plain}
\newtheorem{theorem}{Theorem}[section]

\newtheorem{lemma}[theorem]{Lemma}

\theoremstyle{definition}
\newtheorem{definition}[theorem]{Definition}
\newtheorem{example}[theorem]{Example}

\newtheorem{remark}[theorem]{Remark}

\theoremstyle{remark}

\newcommand\RR{\mathbb{R}}

\newcommand\by{\boldsymbol{y}}
\newcommand\be{\boldsymbol{e}}

\newcommand\bk{\boldsymbol{k}}
\newcommand\bw{\boldsymbol{w}}

\newcommand\bW{\boldsymbol{W}}
\newcommand\bu{\boldsymbol{u}}

\newcommand\bx{\boldsymbol{x}}
\newcommand\bv{\boldsymbol{v}}
\newcommand\bz{\boldsymbol{z}}
\newcommand\bY{\boldsymbol{Y}}
\newcommand\bA{\boldsymbol{A}}
\newcommand\bB{\boldsymbol{B}}

\newcommand{\defi}{\textbf}

\usepackage{tikz, graphicx}
\usepackage[margin=1.5cm]{caption}

	\usetikzlibrary{positioning, arrows, arrows.meta}
	\usepackage{pgfplots}
        \pgfplotsset{compat=1.16}
        \tikzset{%
        fwdrxn/.style={very thick, arrows={-Stealth[length=5pt,width=5pt]}},
        revrxn/.style={very thick, arrows={-Stealth[length=5pt,width=5pt,left]}},
        newt/.style={turq, opacity=0.15}
        }
        \tikzset{near start abs-right/.style={xshift=1cm}}
        \tikzset{near start abs-left/.style={xshift=-3.5cm}}
        \tikzset{near start abs-up/.style={yshift=1.5cm}}
        \tikzset{near start abs-down/.style={yshift=-1cm}}

    \definecolor{viridisyellow}{RGB}{253,231,36}
    \definecolor{viridisyellowpale}{RGB}{239,223,81}
    \definecolor{viridisgreen}{RGB}{121,209,81}
        \definecolor{hlgreen}{RGB}{16,115,16}
    \definecolor{viridisturq}{RGB}{34,167,132}
    \definecolor{viridisblue}{RGB}{64,67,135}
    \definecolor{viridisviolet}{RGB}{68,1,84}

    \definecolor{magmapink}{RGB}{188,81,119}
    \definecolor{pastelpink}{RGB}{253,191,210}
    
	\definecolor{ratecnst}{RGB}{172,172,172}
	
\usepackage{multirow}
\usepackage{array}

\usepackage[normalem]{ulem}	

\begin{document}

\title{Weakly reversible single linkage class realizations of polynomial dynamical systems: an algorithmic perspective}

\author[1]{
         Gheorghe Craciun%
}
\author[2]{
        Abhishek Deshpande%
}
\author[3]{
        Jiaxin Jin%
}
\affil[1]{\small Department of Mathematics and Department of Biomolecular Chemistry, University of Wisconsin-Madison}
\affil[2]{Center for Computational Natural Sciences and Bioinformatics, \protect \\
 International Institute of Information Technology Hyderabad}
\affil[3]{\small Department of Mathematics, The Ohio State University}

\date{}

\maketitle

\begin{abstract}

Systems of differential equations with polynomial right-hand sides are very common in applications. In particular, when restricted to the positive orthant, they appear naturally (according to the law of {\em mass-action kinetics}) in ecology, population dynamics, as models of biochemical interaction networks, and models of the spread of infectious diseases. On the other hand, their mathematical analysis is very challenging in general; in particular, it is very difficult to answer questions about the long-term dynamics of the variables (species) in the model, such as questions about persistence and extinction. Even if we restrict our attention to mass-action systems, these questions still remain challenging. On the other hand, if a polynomial dynamical system has a {\em weakly reversible single linkage class} ($W\!R^1$) {\em realization}, then its long-term dynamics is known to be remarkably robust: all the variables are persistent (i.e., no species goes extinct), irrespective of the values of the parameters in the model. Here we describe an algorithm for finding $W\!R^1$ realizations of polynomial dynamical systems, whenever such realizations exist. 
\end{abstract}

\section{Introduction}
\label{sec:intro}

By a system of differential equations with polynomial right-hand sides (or simply a \emph{polynomial dynamical system}), we mean a dynamical system of the form
\begin{equation}\label{eq:poly-intro}
\begin{split}
    \frac{dx_1}{dt} &= p_1(x_1, \ldots, x_n), \\ 
    \frac{dx_2}{dt} &= p_2(x_1, \ldots, x_n), \\ 
                    &\qquad \quad \vdots  \\
    \frac{dx_n}{dt} &= p_n(x_1, \ldots, x_n), \\ 
\end{split}
\end{equation}
where each $p_i(x_1,\ldots, x_n)$ is a polynomial in the variables  $x_1,\ldots, x_n$. In general, such systems are very difficult to analyze due to nonlinearities and feedbacks that may give rise to bifurcations, multiple basins of attraction, oscillations, and even chaotic dynamics. The second part of Hilbert's 16th problem (about the number of limit cycles of polynomial dynamical systems in the plane) is still essentially unsolved, even for {\em quadratic} polynomials~\cite{Ilyashenko2002}. Even the simplest object associated to \eqref{eq:poly-intro}, its steady state set, can give rise to highly nontrivial questions in real algebraic geometry.
 
Polynomial dynamical systems  show up very often as standard models (based on {\em mass-action kinetics}) in biology, chemistry,  population dynamics, infectious disease models, and many other areas of applications. In such models the variables $x_i$  represent populations, concentrations, or other quantities that cannot become negative, so the domain of \eqref{eq:poly-intro} is restricted to the positive orthant.  For example, in a biochemical network we may have the reaction $X_1+X_2 \to X_3$, which consumes $X_1$ and $X_2$ and produces $X_3$; according to mass-action kinetics, this reaction contributes a negative monomial term of the form ``$-kx_1x_2$" on the right-hand side of $\frac{dx_1}{dt}$ and $\frac{dx_2}{dt}$, and a positive monomial term  ``$kx_1x_2$" on the right-hand side of $\frac{dx_3}{dt}$, where $x_1, x_2, x_3$ denote the concentrations of the chemical species $X_1, X_2, X_3$. The  parameter $k$ is called {\em reaction rate constant}.
A {\em reaction network} consists of a set of such reactions, and if we add all these terms for all the reactions in the network (each one with its own reaction rate constant) we obtain  standard dynamical system models for the network. In general, one cannot just rely on numerical simulations to deduce the dynamical properties of these models, because the values of the reaction rate constants cannot usually be estimated accurately.  Therefore, it becomes very important to relate the {\em structural properties} of the reaction network with {\em dynamical properties} that can be generated by it~\cite{yu2018mathematical, CraciunDickensteinShiuSturmfels2009, craciun2008identifiability, pantea2012persistence, gopalkrishnan2014geometric, craciun2020efficient, feinberg1979lectures, feinberg2019foundations}.  


Alternatively, one may start with a system of the form \eqref{eq:poly-intro} obtained from fitting some experimental data, with little or no information on the generating reaction network. In general, if a polynomial dynamical system is generated by some reaction network, then there are actually  infinitely many other networks that also generate it~\cite{craciun2008identifiability}. This {\em lack of unique identifiability} of an underlying network can actually be leveraged to analyze the dynamics of a system of the form \eqref{eq:poly-intro}: if a network with certain properties can be found to generate it, then we may be able to immediately infer its dynamical behavior. 

Some of the most important  questions for polynomial systems \eqref{eq:poly-intro} are related to the {\em long-term dynamics} of its solutions, which is usually analyzed in terms of the mathematical properties of {\em persistence} and {\em permanence}. The property of \emph{persistence} means that no species can ``go extinct", i.e., for any solution $\bx(t)$ of the system, we have $\displaystyle{\lim\inf}_{t\to\infty}\bx_i(t) > 0$ for all species $i$. The (stronger) property of \emph{permanence} means that the system has a globally attracting compact set. 

A class of networks whose long-term dynamics is best understood is the family of  {\em weakly reversible single linkage class networks}~\cite{yu2018mathematical}; here we call them simply ``{\em $W\!R^1$ networks}", and we will refer to polynomial systems \eqref{eq:poly-intro} that have $W\!R^1$ realizations as ``{\em $W\!R^1$ systems}". 
Specifically, $W\!R^1$ systems have been shown to be  persistent and permanent in a very robust way, which even allows for the explicit construction of globally attracting invariant sets~\cite{gopalkrishnan2014geometric, boros2020permanence}.  Moreover, complex balanced $W\!R^1$ systems have been shown to be {\em globally stable}, i.e., they have  a {\em globally attracting point} within each linear invariant subspace~\cite{anderson2011proof}.

Not only are the long-term dynamical properties of $W\!R^1$ systems well understood, but also their persistence and permanence properties hold {\em for any choices of parameter values}, in a sense that will be made clear below. This fact is very important in applications because the exact values of the coefficients in the polynomial right-hand sides of these dynamical systems are often very difficult to estimate accurately. 

In this paper, we describe an efficient algorithm for determining whether a given polynomial dynamical system \eqref{eq:poly-intro} admits a $W\!R^1$ realization, and for finding such a realization whenever it exists. 

\bigskip

\textbf{Structure of the paper.}
In Section~\ref{sec:reaction_networks}, we introduce some basic terminology of reaction networks. 
Primarily, we present the notion of \emph{net reaction vectors}, which play a key role in the main algorithm. 
In Section~\ref{sec:algorithms}, we propose Algorithm~\ref{algorithm 1} to find if there exists a weakly reversible reaction network consisting of a single connected component that generates a given dynamical system. 
In Section~\ref{sec:special case and implementation}, we discuss some special cases of weakly reversible realizations with a single linkage class and go through the steps in Algorithm~\ref{algorithm 1} using several examples.
Moreover, we illustrate how to implement this algorithm in practice.
In Section~\ref{sec:discussion}, we summarize our findings in this paper and outline directions for future work.

\medskip

\textbf{Notation.}
We denote by $\mathbb{R}_{\geq 0}^n$ and $\mathbb{R}_{>0}^n$ the set of vectors in $\mathbb{R}^n$ with non-negative and positive entries respectively. Given two vectors $\bx\in \mathbb{R}_{>0}^n$ and $\by \in \RR^n$, we use the following notation for a monomial with exponents given by $\by$:
\begin{equation} \notag
\bx^{\by} = x_1^{y_{1}} \ldots x_n^{y_{n}},
\end{equation}
where $\bx = (x_1, \ldots, x_n)^{\intercal}$ and $\by = (y_1, \ldots, y_n)^{\intercal}$.

\section{Reaction networks}
\label{sec:reaction_networks}

\begin{definition} 
A \defi{reaction network}, also called a \defi{Euclidean embedded graph (E-graph)}, is a directed graph $G = (V, E)$ in $\RR^n$, where $V \subset \mathbb{R}^n$ is a finite set of \defi{vertices}, $E \subseteq V \times V$ represents the set of \defi{edges}, and such that there are neither self-loops nor isolated vertices in $G$.
We denote the {number of vertices} by $m$, and let $V = \{ \by_1, \ldots, \by_m \}$.  
A directed edge $(\by_i, \by_j) \in E$ represents a \textbf{reaction} in the network, and is also denoted by $\by_i \rightarrow \by_j$. Moreover, we define the \defi{reaction vector} associated with the edge $\by_i \rightarrow \by_j$ as $\by_j - \by_i \in\mathbb{R}^n$. Here $\by_i$ and $\by_j$ denote the \defi{source vertex} and \defi{target vertex} respectively.

\end{definition}

\begin{definition}
Let $G=(V, E)$ be a Euclidean embedded graph.
The \defi{stoichiometric subspace} of $G$ is the vector space spanned by the reaction vectors as follows:
\begin{equation} \notag
S = \rm{span} \{\by' - \by\, |\, \by \rightarrow \by' \in E \}.
\end{equation}
Moreover, for any positive vector $\bx_0 \in\mathbb{R}_{>0}^n$, the affine polyhedron $(\bx_0 + S ) \cap \mathbb{R}^n_{>0 }$ is called the \defi{stoichiometric compatibility class} of $\bx_0$.
\end{definition} 

\begin{definition} 
Let $G=(V, E)$ be a Euclidean embedded graph.

\begin{enumerate}[label=(\alph*)]
\item The set of vertices $V$ is partitioned by its \defi{connected components} (also called \defi{linkage classes}), which correspond to the subset of vertices belonging to that connected component. 

\item A connected component $L \subseteq V$ is \defi{strongly connected}, if every edge is part of an oriented cycle. 
Further, a strongly connected component $L \subseteq V$ is called a \defi{terminal strongly connected component} if for every vertex $\by\in L$ and $\by\rightarrow\by'\in E$, we have $\by'\in L$.

\item $G=(V, E)$ is said to be \defi{weakly reversible}, if every connected component is strongly connected, i.e., every edge is part of an oriented cycle.
\end{enumerate}
\end{definition}

\begin{remark}
For a weakly reversible reaction network $G=(V, E)$, every vertex $\by \in V$ is a source and a target vertex.
Furthermore, every linkage class is a strong linkage
class, as well as a terminal strong linkage class.
\end{remark}

\begin{definition}
Let $G = (V, E)$ be a Euclidean embedded graph, with $m$ vertices and $\ell$ connected components. Suppose the dimension of the stoichiometric subspace $S$ is $s = \dim (S)$, then the \defi{deficiency} of the network $G$ is the non-negative integer defined as follows:
\begin{equation*}    
\delta = m - \ell - s.
\end{equation*}
\end{definition}

\begin{definition}
Consider a Euclidean embedded graph $G=(V, E)$, and denote by $V_S\subseteq V$ the set of source vertices in $G$. 
Then $G$ is said to be \defi{endotactic}~\cite{craciun2013persistence, anderson2020classes,craciun2020endotactic}, if for every $\bv\in\mathbb{R}^n$ and $\by\to \by'\in E$ satisfying $\bv\cdot (\by'-\by)<0$, there exists $\tilde{\by}\to \hat{\by} \in E$, such that 
\begin{equation*}
\bv\cdot (\hat{\by} - \tilde{\by}) > 0, \ \text{and } \ \bv \cdot \tilde{\by} < \bv \cdot \by.
\end{equation*}
Moreover, $G$ is said to be \defi{strongly endotactic}~\cite{gopalkrishnan2014geometric, anderson2020classes}, if for every $\bv \in\mathbb{R}^n$ and $\by\to \by'\in E$ satisfying $\bv\cdot (\by'-\by)<0$, there exists $\tilde{\by} \to \hat{\by} \in E$, such that for every ${\by}^{*} \in V_S$,
\begin{equation*}
\bv \cdot (\hat{\by} - \tilde{\by})>0, \
\bv\cdot \tilde{\by} < \bv\cdot \by, \ \text{and } \
\bv \cdot \tilde{\by} \leq \bv \cdot {\by}^{*}.
\end{equation*}
\end{definition}

\begin{remark}[\cite{anderson2020classes}]
It can be shown that weakly reversible reaction networks are endotactic. Furthermore, if a network is weakly reversible and consists of a single linkage class, then it is strongly endotactic.
\end{remark}

\begin{figure}[H]
\centering
\includegraphics[scale=0.5]{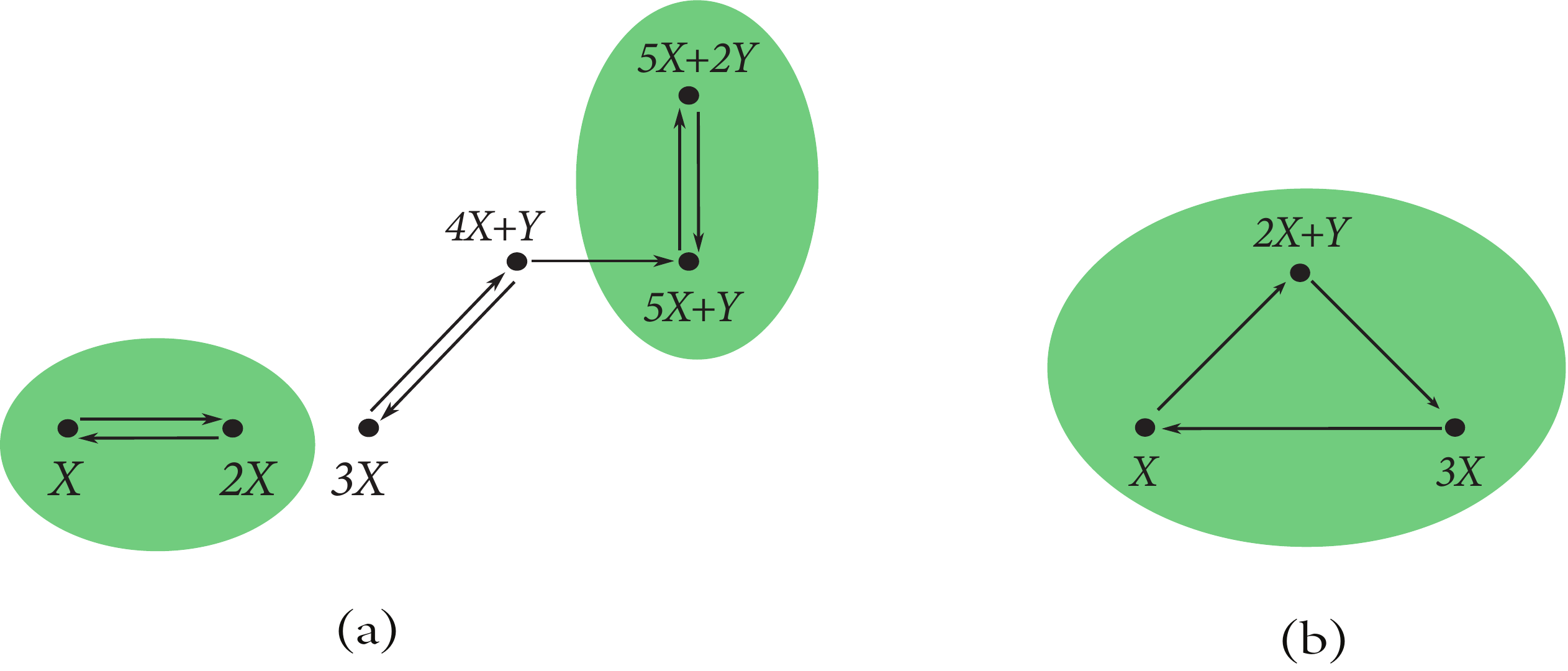}
\caption{\small (a) This reaction network has two linkage classes, and two terminal strongly connected components (shown in the green shaded region).  It has a stoichiometric subspace of dimension 2 and the deficiency $\delta = m - \ell -s = 6- 2- 2=2$. (b) This reaction network is weakly reversible and has one terminal strongly connected component. It has a stoichiometric subspace of dimension 2 and the deficiency $\delta = m - \ell -s = 3- 1-2=0$.}
\label{fig:euclidean_graphs}
\end{figure} 

Figure~\ref{fig:euclidean_graphs} shows two examples of reaction networks. 
A reaction network can generate a wide range of dynamical systems. We are interested in mass-action kinetics, which has been extensively studied in ~\cite{yu2018mathematical,feinberg1979lectures,voit2015150,guldberg1864studies,gunawardena2003chemical,adleman2014mathematics}.

\begin{definition} 
Let $G=(V,E)$ be a Euclidean embedded graph, we denote the \defi{vector of reaction rate constants} by
$\bk = (k_{\by_i \rightarrow \by_j})_{\by_i \rightarrow \by_j \in E} \in \mathbb{R}_{>0}^{E}$, and $k_{\by_i \rightarrow \by_j}$ or $k_{ij}$ is called the \defi{reaction rate constant} on the edge $\by_i \rightarrow \by_j$. 
The \defi{associated mass-action system} generated by $(G, \bk)$ 
on $\RR^n_{>0}$ is given by:
\begin{equation} \label{eq:mass_action}
\frac{\mathrm{d} \bx}{\mathrm{d} t}= \sum_{\by_i \rightarrow \by_j \in E}k_{\by_i \rightarrow \by_j} \bx^{\by_i}(\by_j - \by_i).
\end{equation}
\end{definition}

\begin{definition} \label{def:steady state}
Consider the associated mass-action system generated by 
$(G, \bk)$ in \eqref{eq:mass_action}. A point $\bx^* \in \mathbb{R}_{>0}^n$ is called a \defi{positive steady state} if 
\begin{equation} \label{eq:ss}
\frac{\mathrm{d} \bx}{\mathrm{d} t} 
= \sum_{\by_i \rightarrow \by_j \in E}k_{\by_i \rightarrow \by_j} (\bx^*)^{\by_i}(\by_j - \by_i)
= \mathbf{0}.
\end{equation}
\end{definition}

It is well known that every mass-action system admits a matrix decomposition~\cite{horn1972general}.
Hence, we can illustrate the mass-action system \eqref{eq:mass_action} in the following vectorial representation:
\begin{equation} \label{eq:mass_action_horn_jackson}
\frac{d\bx}{dt} = \bY \bA_{\bk} {\bx}^{\bY},
\end{equation}
where $\bY$ is a matrix whose columns are the vertices, defined as
\begin{equation*}
\bY = (\by_1, \ \by_2, \ \ldots, \ \by_m),
\end{equation*}
and $\bA_{\bk}$ is the negative transpose of the graph Laplacian of $(G, \bk)$, defined as
\begin{equation*}
[\bA_{\bk}]_{ji} =
\begin{cases}
k_{\by_i\rightarrow\by_j}, & \text{if } \by_i\rightarrow\by_j \in E, \\[5pt]
- \sum\limits_{\by_i\rightarrow\by_j \in E} k_{\by_i\rightarrow\by_j}, & \text{if } i=j, \\[5pt]
0  & \rm{otherwise},
\end{cases}
\end{equation*}
and ${\bx}^{\bY}$ is the vector
of monomials given by
\begin{equation*}
{\bx}^{\bY} = ({\bx}^{\by_1},{\bx}^{\by_2}, \ldots, {\bx}^{\by_m})^{\intercal}.
\end{equation*}
In general, $\bY$ is called the \defi{matrix of vertices}, and $\bA_{\bk}$ is called the \defi{Kirchoff} matrix. 

\medskip

Here, we list one of the most important properties of the Kirchoff matrix $\bA_{\bk}$.

\begin{theorem}[\cite{feinberg1977chemical}]
\label{thm:supp_terminal_linkage}
Let $(G, \bk)$ be a mass-action system, and $T_1, T_2, \ldots, T_t$ be the terminal strongly connected components of $G$. Then there exists a basis $\{\be_1,\be_2,\ldots,\be_t\}$ for $\ker (\bA_{\bk})$, such that 
\begin{equation*}
\be_p =
\begin{cases} 
    \begin{array}{cl}
         [\be_p]_i  > 0, & \text{ if } \by_i \in T_p, \\[5pt]
         [\be_p]_i  = 0, & \text{ otherwise.}
    \end{array} 
\end{cases}
\end{equation*}
\end{theorem}

\begin{example} \label{ex:terminal_kernel}

We will revisit the network shown in Figure~\ref{fig:euclidean_graphs}(a) to verify Theorem \ref{thm:supp_terminal_linkage}. First, we set all vertices in the network as follows:
\begin{equation*}
\begin{split}
& X \equiv \by_1 = (1, 0)^{\intercal}, \ \ 2X \equiv \by_2 = (2, 0)^{\intercal}, \ \ 3X \equiv \by_3 = (3, 0)^{\intercal}, 
\\& 4X+Y \equiv \ \by_4 = (4, 1)^{\intercal}, \ \ 5X+Y \equiv \by_5 = (5, 1)^{\intercal}, \ \ 5X+2Y \equiv \by_6 = (5, 2)^{\intercal}.
\end{split}
\end{equation*}
The Kirchoff matrix of the network is given by:
\begin{equation} \notag
\bA_{\bk}=
\begin{bmatrix}
-k_{12} & k_{21} & 0 & 0 & 0 & 0 \\
k_{12} & -k_{21} & 0 & 0 & 0 & 0 \\
0 & 0 & -k_{34} & k_{43} & 0 & 0 \\
0 & 0 & k_{34} & - k_{43} - k_{45} & 0 & 0 \\
0 & 0 & 0 & k_{45} & -k_{56} & k_{65} \\
0 & 0 & 0 & 0 & k_{56} & -k_{65} 
\end{bmatrix}.
\end{equation}
Using a direct computation, the following vectors form a basis for $\rm{ker}(\bA_{\bk})$:
\begin{equation} \notag
\be_1 = (k_{21}, k_{12}, 0, 0, 0, 0)^{\intercal}, \ \
\be_2 = (0, 0, 0, 0, k_{65}, k_{56})^{\intercal}.
\end{equation}
Thus, we have
\begin{equation*}
\text{supp} (\be_1) = \{1, 2\}, \ \text{and } \
\text{supp} (\be_2) = \{5, 6\}.
\end{equation*}
It is clear that the supports of two basis vectors relate to two terminal strongly connected components $\{X, 2X\}$ and $\{5X+Y, 5X+2Y\}$.
\end{example}

Motivated by the matrix decomposition in \eqref{eq:mass_action_horn_jackson}, we introduce a crucial concept: net reaction vector, and another matrix decomposition in terms of net reaction vectors, which play an important role in finding a realization.

\begin{definition} \label{defn:net_reaction_vector}
Consider a mass-action system $(G, \bk)$, and let $V_S = \{ \by_1, \by_2, \ldots, \by_{m_s}\} \subseteq V$ be the set of source vertices of $G$. For each source vertex $\by_i\in V_S$, the \defi{net reaction vector} $\bw_i$ corresponding to $\by_i$ is given by:
\begin{equation} 
\bw_i = \sum\limits_{\by_i\rightarrow\by_j\in E}k_{\by_i\rightarrow\by_j}(\by_j - \by_i),
\end{equation}
Moreover, we denote the \defi{matrix of net reaction vectors} as follows:
\begin{equation}
\bW = \left(\bw_1, \ \bw_2, \ \ldots, \ \bw_{m_s} \right).
\end{equation}
\end{definition}

Following Definition \ref{defn:net_reaction_vector}, for each source vertex 
$\by_i\in V_S$, we can rewrite the corresponding net reaction vector $\bw_i$ as 
\begin{equation}
\bw_i = \sum\limits_{\by_i\rightarrow\by_j\in E}k_{\by_i\rightarrow\by_j} \by_j 
- \left(\sum\limits_{\by_i\rightarrow\by_j\in E}k_{\by_i\rightarrow\by_j} \right) \by_i.
\end{equation}
Using a direct computation, we can rewrite the matrix decomposition in \eqref{eq:mass_action_horn_jackson} as
\begin{equation} \label{eq:mass_action_net_reaction}
\frac{d\bx}{dt} = \bW {\bx}^{\bY_s},
\end{equation}
where ${\bx}^{\bY_s}$ is the vector
of monomials given by
\begin{equation*}
{\bx}^{\bY_s} = ({\bx}^{\by_1},{\bx}^{\by_2}, \ldots, {\bx}^{\by_{m_s}})^{\intercal}.
\end{equation*}
Further, we let $\bY_s = (\by_1, \ \by_2, \ \ldots, \ \by_{m_s})$ denote the \defi{matrix of source vertices}, whose columns are the source vertices.

\medskip

The following Lemma concerns the matrix of net reaction vectors when the mass-action system is weakly reversible.

\begin{lemma} \label{lem:wr_ker}
Consider a weakly reversible mass-action system $(G,\bk)$ with vertices $\{\by_i\}_{i=1}^m$ and stoichiometric subspace $S$. Let $\{\bw_i\}_{i=1}^m$ be the net reaction vectors of $G$, and $\bW = (\bw_1, \bw_2, \ldots, \bw_{m} )$ be the matrix of net reaction vectors. Then we have
\begin{equation} \label{lem:image_stoich} 
\rm{Im}(\bW) = S.
\end{equation} 
\end{lemma}

\begin{proof}
It is clear that $\rm{Im}(\bW) \subseteq S$ from Definition \ref{defn:net_reaction_vector}.
Suppose that $\rm{Im}(\bW)\subset S$, then there exists a non-zero vector $\bv$, such that
\begin{equation} \notag
\bv \in S, \ \text{and } \ \bv \perp \bW.
\end{equation}
Since $\mathbf{0} \neq \bv \in S$, there exists a reaction $\by_i \to \by_j \in E$ such that $\bv \cdot (\by_j - \by_i) \neq 0$. This implies that the set $\{ \bv \cdot \by_i \}_{i=1}^m$ has at least two different numbers. Now let $V_{\max}$ be the subset of vertices which maximizes the dot product as follows.
\begin{equation} \notag
V_{\max} = \{ \by_i \in V: \ \bv \cdot \by_i = \max_j ( \bv \cdot \by_j ) \}.
\end{equation}
Since $G$ is weakly reversible, there exists an edge from a vertex in $V_{\max}$ to a vertex not belonging to it. Without loss of generality, let $\by_1 \in V_{\max}$, and $\by_1 \to \by_2 \notin V_{\max}$ be this edge. Note that for all $i=1$, $2,\ldots, m$, we have $\bv \cdot (\by_i - \by_1) \leq 0$. Thus, we obtain
\begin{equation} \notag
\bv \cdot \bw_1 
= \sum_{\by_j \in V} k_{\by_1 \to \by_j} \bv \cdot (\by_j - \by_1) \leq k_{\by_1 \to \by_2} \bv \cdot (\by_2 - \by_1) < 0. 
\end{equation}
This contradicts with $\bv \perp \bW$, and the result follows.
\end{proof}

At the end of this section, we introduce some important dynamical properties. 

\begin{definition}
Let $(G, \bk)$ be a mass-action system. Then $(G, \bk)$ is called \defi{persistent}, if every solution $\bx(t)$ with initial condition $\bx(0) \in \mathbb{R}^n_{>0}$ satisfies the following:
\begin{equation} \notag
\displaystyle\liminf_{t\to\infty}\bx_i(t) >0, \ \text{for } i=1,2,\ldots,n.
\end{equation}
\end{definition}

\begin{definition}
Let $(G, \bk)$ be a mass-action system. Then $(G, \bk)$ is called \defi{permanent}, if given any stoichiometric compatibility class $A$ and any solution $\bx(t)$ with initial condition $\bx(0)\in A$,
there exists a time $T$ and a compact set $D\subset A$, such that for all $t > T$,
\begin{equation*}
\bx(t)\in D.
\end{equation*}
\end{definition}

\begin{definition}
Let $(G, \bk)$ be a mass-action system. 
A point $\tilde{\bx}\in\mathbb{R}^n_{>0}$ is said to be a \defi{global attractor} within its stoichiometric compatibility class, if $\displaystyle\lim_{t\to\infty}\bx(t) = \tilde{\bx}$.
\end{definition}

\section{Main result}
\label{sec:algorithms}

The goal of this section is to present the main algorithm of this paper: Algorithm~\ref{algorithm 1}, which searches for the existence of a weakly reversible realization consisting of a single linkage class. In particular, this algorithm outputs a maximal realization, whenever it exists. The input of Algorithm~\ref{algorithm 1} is the matrix of source vertices $\bY_s = (\by_1,\by_2,\ldots,\by_m)$, and the matrix of net reaction vectors $\bW = (\bw_1,\bw_2,\ldots,\bw_m)$.

\subsection{Algorithm for weakly reversible realization with a single linkage class}

Here, we sketch the key idea behind Algorithm~\ref{algorithm 1}: given a weakly reversible realization $G = (V, E)$ with a single linkage class, adding new reactions among the vertices in $V$ on the realization $G$ preserves the properties of weak reversibility and single linkage class. We present this algorithm below and give proof of its correctness.

\begin{algorithm}
\caption{(Check the existence of a weakly reversible realization with a single linkage class)} \label{algorithm 1}

\begin{algorithmic}[1]

\smallskip

\REQUIRE The matrix of source vertices $\bY_s = (\by_1, \ldots, \by_m)$, and the matrix of net reaction vectors $\bW = (\bw_1, \ldots, \bw_m)$ that generate the dynamical system $\dot{\bx} = \sum\limits_{i=1}^m \bx^{\by_i} \bw_i$. 

\smallskip 

\ENSURE Either return a weakly reversible realization consisting of a single linkage class, or print that such a realization does not exist.

\smallskip

\FOR{$i = 1, 2, \ldots, m$}

\State Define the matrix $\bB_i \in \mathbb{R}_{n \times m}$, with $k^{\rm th}$ column $B_{i, k} := (\by_k -\by_i)$ for $1 \leq k \leq m$. \label{alg: Bi}
 
\IF{there exists a vector $\bv^* = (\bv^*_1, \ldots, \bv_m^*) \in \mathbb{R}^{m}_{\geq 0}$, such that $\bB_i \bv^* = \bw_i$} \label{algo_1_line:2}

\State Set $\bv^*_i = 1$.

\ELSE

\STATE {\bf Print:} There is no realization. {\bf Exit.}

\ENDIF

\State Set $S_i = \text{supp} (\bv^*)$.

\FOR{$j = 1, 2, \dots, m$} \label{algo_1_line:9}

\IF{$j \in S_i$} 
\State {\bf Continue}

\ELSE

\IF{there exists a vector $\bv \in \mathbb{R}^{m}_{\geq 0}$, such that $\bB_i \bv = \bw_i$ and $\bv_j > 0$} 

\State $S_i = S_i \cup \text{supp} (\bv)$.

\ENDIF

\ENDIF

\algstore{bkbreak}

\end{algorithmic}

\begin{algorithmic}[1]

\algrestore{bkbreak}

\ENDFOR \label{algo_1_line:17}

\ENDFOR

\State Define vector $r_i := (r_{i,1}, r_{i,1}, \cdots, r_{i,m})^{\intercal}$, with
\begin{equation} \label{alg: rij}
    r_{i,j} =
    \begin{cases}
    1, & \text{for } j \in S_i, j \neq i, \\[5pt]
    0, & \text{for } j \notin S_i, \\[5pt]
    - \sum\limits_{l \neq i} r_{i,l}, & \text{for } j = i.
    \end{cases}
\end{equation}

\State Collect $\{r_i\}^m_{i=1}$ and construct the Kirchoff matrix $Q = (r_1,r_2,\ldots,r_m) \in \mathbb{R}_{m \times m}$. \label{algo_1_line:7}

\IF{$\dim (\ker(Q)) = 1$ and $\text{supp} (\ker(Q)) = \{1, \ldots, m \}$}  \label{algo_1_line:8}

\STATE {\bf Print:} There exists a weakly reversible realization  with a single linkage class.

\STATE {\bf Print:} The vertices of this realization are given by $V = \{ \by_1, \ldots, \by_m \}$.

\STATE {\bf Print:} 
The edges of this realization are given by $E = \{ \by_i \to \by_j: r_{i,j} > 0\}$.

\ELSE

\STATE {\bf Print:} There is no realization.

\ENDIF

\end{algorithmic}
\end{algorithm}

\newpage

\tikzset{%
    Node/.style={rectangle, rounded corners, draw=black, thick, fill=blue!10, fill opacity = 1, minimum width=4.5cm, minimum height=1cm, outer sep=0pt},
    Edge/.style={very thick, style={very thick, arrows={-Stealth[length=7.5pt,width=7.5pt]}}},
}

\begin{figure}[H]
\centering
    \begin{tikzpicture}
    \node[Node,  minimum width=5.75cm, fill=orange!30] (fail1) at (1,-3.75) {\small No realization};
    \node[Node, minimum height=2cm, minimum width=5.75cm] (input) at (9,-0.25) {};
        \node at (7.1,0) [right] {\small $\bY_s = \begin{pmatrix} \by_1 ,  \ldots ,  \by_m\end{pmatrix} \in  \mathbb{R}_{n \times m}$};
        \node at (7,-0.5) [right] {\small $\bW = \begin{pmatrix} \bw_1 ,  \ldots ,  \bw_m\end{pmatrix} \in \mathbb{R}_{n \times m}$};
    \node[Node,  minimum width=7cm, minimum height=1.5cm] (realization) at (9,-3.75) {};
        \node at (9,-3.5) {\small Check if there exists a vector $\bv^*\in \mathbb{R}^{m}_{\geq 0}$, };
        \node at (9,-4) {\small such that $\bB_i \bv^* = \bw_i$ where $B_{i, k} = \by_k -\by_i$?};
    \draw[Edge] (input)--(realization) node [midway, left] {\small \ttfamily\bfseries \textcolor{blue}{for} i = 1,\ldots,m};
    \draw[Edge] (realization)node[xshift=-3.9cm, above] {\small no}--(fail1);
    \node[Node,  minimum width=7cm, minimum height=1cm] (S) at (9,-6) {};
        \node at (9,-6) {\small Set $\bv^*_i=1$, and $S_i = \text{supp} (\bv^*)$};
    \draw[Edge] (realization)--(S) node [midway, right] {\small yes};
    \node[Node, minimum height=1cm, minimum width=7cm] (new) at (9,-8.5) {};
        \node at (9,-8.5) {\small Check if $j \in S_i$?};
        \draw[Edge] (S)--(new) node [midway, right] {\small \ttfamily\bfseries \textcolor{blue}{for} j = 1,\ldots,m};  \draw[Edge] (new) to node[xshift=0.1cm, above] {\small yes} (4.5,-8.5) 
        to (4.5,-7.25); 
    \node[Node, minimum height=2cm, minimum width=7cm] (Update) at (9,-11) {};
        \node at (9,-10.4) {\small Check if there exists a vector $\bv \in \mathbb{R}^{m}_{\geq 0}$, };
        \node at (9,-11) {\small such that $\bB_i \bv = \bw_i$ and $\bv_j > 0$?};
        \node at (9,-11.6) {\small If yes, $S_i = S_i \cup \text{supp} (\bv)$};
        \draw[Edge] (Update) to node[xshift=0cm, above] {} (9, -12.4) to (3,-12.4)
        to (3,-7.25) to (9, -7.25);
    \draw[Edge] (new)--(Update) node [midway, right] {\small no};; 
        \draw[Edge] (Update) to node[right] {\small } (9,-13.2) 
        to (13,-13.2) to (13,-2.1) to (9, -2.1);
    \node[Node,  minimum height=2.75cm, minimum width=5.75cm] (Q) at (9,-15.5) {};
    \node at (9,-14.75) {\small Construct matrix $Q = (r_1,\ldots,r_m)$.};
    \node at (9,-15.5) {\small Check if $\dim (\ker(Q)) = 1$, and};
    \node at (9,-16.25) {\small $\text{supp} (\ker(Q)) = \{1, \ldots, m \}$?};
    \draw[Edge] (Update)--(Q) node[below] {};
    \draw[Edge] (Q)node[xshift=-3.3cm, above] {\small no}--(1,-15.5) --(fail1);
    \node at (9,-12.4) {$\bullet$};
    \node at (9,-12.4) [right] {\small \ttfamily\bfseries \textcolor{blue}{endfor}\,};
    \node at (9,-13.2) {$\bullet$};
    \node at (9,-13.2) [left] {\small \ttfamily\bfseries \textcolor{blue}{endfor}\,};
    \node[Node,  minimum width=5.75cm, fill=green!30] (success) at (9, -18.5) {\small A weakly reversible realization with a single linkage class found};
    \draw[Edge] (Q)--(success) node [midway, right] {\small yes};

\end{tikzpicture}  
\caption{Algorithm~\ref{algorithm 1} for finding a weakly reversible realization with a single linkage class that generates a given polynomial dynamical system $\dot{\bx} = \sum_{i=1}^m \bx^{\by_i} \bw_i$. 
}
\label{fig:Alg} 
\end{figure}
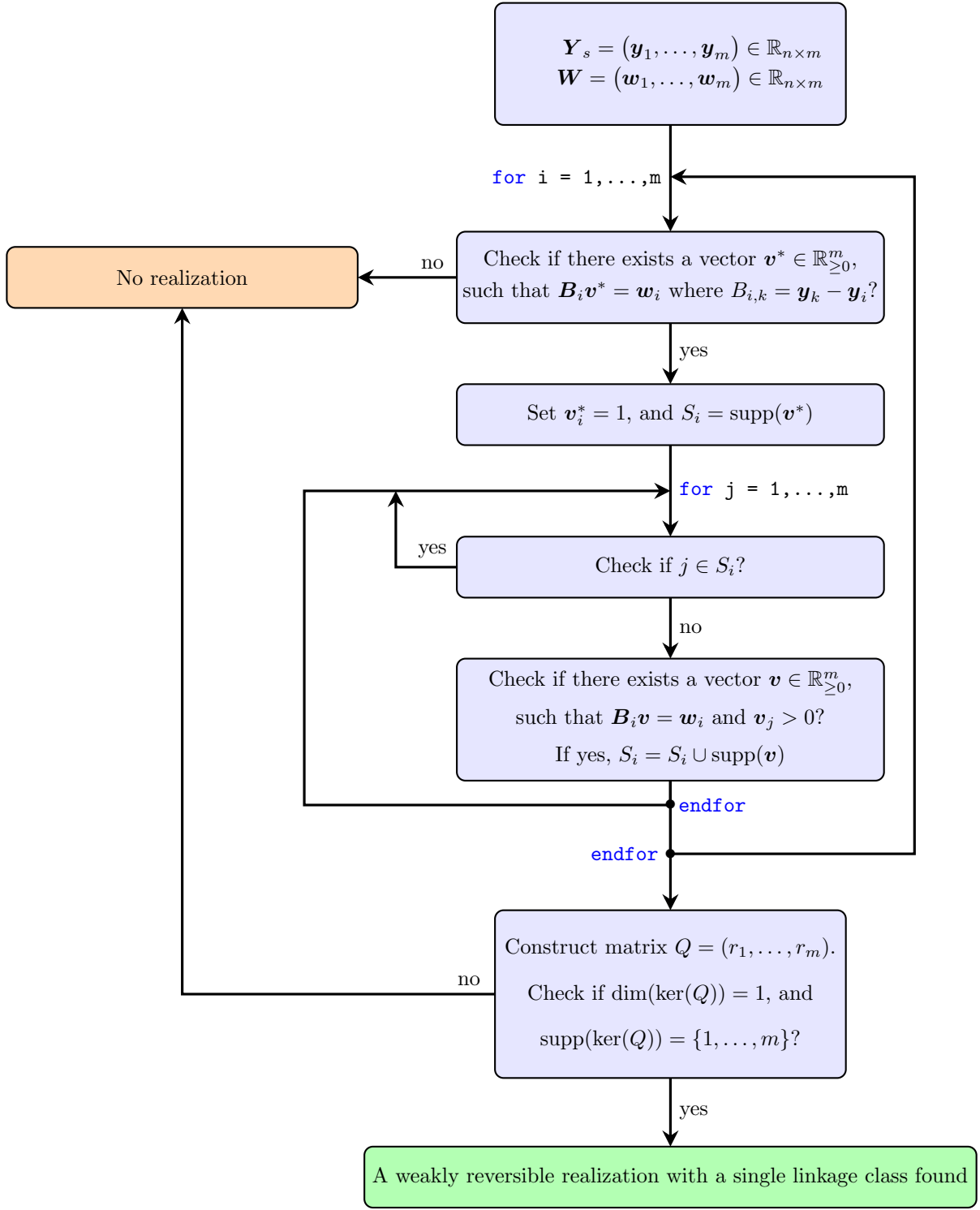

Now we show the correctness of Algorithm~\ref{algorithm 1} via the following two Lemmas.

\begin{lemma}
\label{lem: algorithm 1 part a}

Suppose Algorithm~\ref{algorithm 1} reaches line~\ref{algo_1_line:8} and satisfies the conditions on line~\ref{algo_1_line:8}, then there exists a weakly reversible realization consisting of a single linkage class that generates the dynamical system $\dot{\bx} = \sum\limits_{i=1}^m \bx^{\by_i} \bw_i$.
\end{lemma}

\begin{proof}

Since the algorithm passes the condition on line~\ref{algo_1_line:2}\, for $i = 1, \ldots, m$, 
all net reaction vectors $\{\bw_i\}^m_{i=1}$ can be realized by conical combinations of the vectors $B_{i, k} = (\by_k -\by_i)$, with $1 \leq k \leq m$.


For $i = 1, \ldots, m$, we denote $S_i$ by the union of supports on some vectors $\bv \in \mathbb{R}^{m}_{\geq 0}$, such that 
\begin{equation}
\bB_i \bv = \bw_i.
\end{equation}
Suppose there exist $a_i$ distinct vectors $\bv^1, \ldots, \bv^{a_i} \in \mathbb{R}^{m}_{\geq 0}$, with 
$S_i = \bigcup\limits^{a_i}_{q=1} \text{supp} (\bv^q)$.
Then we consider the following vector:
\begin{equation} \label{tilde v}
    \Tilde{\bv} = \frac{1}{a_i} \sum\limits^{a_i}_{q=1} \bv^q,
\end{equation}
and it satisfies
\begin{equation} \notag
\bB_i \Tilde{\bv} = \bw_i, \ \text{and } \ 
\Tilde{\bv} \in \mathbb{R}^{m}_{\geq 0}.
\end{equation}

Recall that each $v^i$ represents one realization corresponding to the net reaction vector $\bw_i$. Here we choose the vector $\Tilde{\bv}$ in~\eqref{tilde v}, where we have weighted all vectors $\{\bv^q\}^{a_i}_{q=1}$ equally in the graph. Hence, it is clear that the reaction $\by_i \to \by_j$ represented by $j \in S_i$  is included in the realization. Further, we note that scaling the reaction rates neither affects weak reversibility nor the number of linkage classes. Hence, the Kirchoff matrix can be designed from line~\ref{algo_1_line:7}.

Using Theorem~\ref{thm:supp_terminal_linkage}, we know that the kernel of the Kirchoff matrix has a basis consisting of non-negative vectors whose supports are the terminal strongly connected components. Recall that since the algorithm satisfies the condition on line~\ref{algo_1_line:8}, we have 
\begin{equation}
\dim (\ker(Q)) = 1, \ \text{and } \
\text{supp} (\ker(Q)) = \{ 1, \ldots, m \}.
\end{equation}
This implies that all vertices corresponding to the Kirchoff matrix $Q$ are in the same terminal strongly connected component. Therefore, this realization is weakly reversible and consists of a single linkage class. 
\end{proof}

\begin{lemma}
\label{lem: algorithm 1 part b}

Suppose there exists a weakly reversible realization consisting of a single linkage class on $\dot{\bx} = \sum\limits_{i=1}^m \bx^{\by_i} \bw_i$, then Algorithm~\ref{algorithm 1} must satisfy the conditions on lines~\ref{algo_1_line:2} and~\ref{algo_1_line:8}.
\end{lemma}

\begin{proof}

From the existence of a realization, all net reaction vectors $\{\bw_i\}^m_{i=1}$ can be realized by conical combinations of the vectors $(\by_k -\by_i)^m_{k=1}$.
Thus, the algorithm must satisfy the condition on line~\ref{algo_1_line:2}, for $i = 1, \ldots, m$. Now it suffices for us to show
\begin{equation} \label{alg 1 line 8}
\dim (\ker(Q)) = 1, \ \text{and } \
\text{supp} (\ker(Q)) = \{1, \ldots, m \}.
\end{equation}

Here we claim that the realization produced by Algorithm~\ref{algorithm 1} consists of the maximum number of reactions. To realize the system $\dot{\bx} = \sum_{i=1}^m \bx^{\by_i} \bw_i$, we need to find a vector $\bv$ for each vertex $\by_i$, such that 
\begin{equation} \label{alg 1 realization}
\bB_i \bv = \bw_i, \ \text{and } \
\bv \in \mathbb{R}^{m}_{\geq 0}.
\end{equation}
First, for $i = 1, \ldots, m$, we get a vector $\bv^* \in \mathbb{R}^{m}_{\geq 0}$ from line~\ref{algo_1_line:2}, which solves Equation~\eqref{alg 1 realization}. 
After setting $\bv^*_i = 1$, we define the initial support set $S_i$ as follows:
\begin{equation}
S_i = \text{supp} (\bv^*).
\end{equation}
Next, we build an inner loop on $j = 1, \ldots, m$. If $j \in S_i$, this implies that we already incorporated the reaction $\by_i \to \by_j$ in the realization. Otherwise, for each $j \notin S_i$, we further check whether there exists a vector $\bv \in \mathbb{R}^{m}_{\geq 0}$, such that 
\begin{equation} \label{alg 1 realization j>0}
\bB_i \bv = \bw_i, \ \text{and } \ 
\bv_j > 0.
\end{equation}
Once we find such vector $\bv$, we update the set $S_i$ as
\begin{equation}
S_i := S_i \cup \text{supp} (\bv).
\end{equation}
This implies that whenever $j \in \text{supp} (\bv)$, we have $j \in S_i$. 

After going through the whole inner loop, we obtain the complete version of set $S_i$. Then we follow the construction in~\eqref{tilde v}, and it is clear that the reaction $\by_i \to \by_j$ represented by $j \in S_i$ is included in the realization.

Now suppose there is a vector $\hat{\bv}$, solving Equation~\eqref{alg 1 realization} and $\text{supp} (\hat{\bv}) \not\subseteq S_i$. 
This implies that there exists $j \in \text{supp} (\hat{\bv})$ with $j \notin S_i$, a contradiction. Thus, the set $S_i$ contains the \emph{maximal} number of positive entries.

From the claim above and line~\ref{algo_1_line:7}, we deduce that for any realization of the system, all reactions between $\{\by_i\}^m_{i=1}$ are included in the realization given by the Kirchoff matrix $Q$. Note that adding more reactions among the current vertices of a weakly reversible single linkage class network will preserve the properties of weak reversibility and the single linkage class condition. This implies that if there exists a weakly reversible realization consisting of a single linkage class, then the realization generated by $Q$ will also be weakly reversible and consist of a single linkage class. By Theorem~\ref{thm:supp_terminal_linkage}, we conclude~\eqref{alg 1 line 8}.
\end{proof}

The following remark is a direct consequence of Lemma~\ref{lem: algorithm 1 part b}.

\begin{remark}
If Algorithm~\ref{algorithm 1} fails at lines~\ref{algo_1_line:2} or~\ref{algo_1_line:8}, then $\dot{\bx} = \sum\limits_{i=1}^m \bx^{\by_i} \bw_i$ does not admit a weakly reversible realization with a single linkage class.
\end{remark}

\section{Special cases and the implementation of Algorithm~\ref{algorithm 1}}
\label{sec:special case and implementation}

After showing Algorithm~\ref{algorithm 1}, we focus on some special cases and the implementation of the algorithm. We will discuss various properties of weakly reversible realizations consisting of a single linkage class but having different deficiencies, and the corresponding implementation of the algorithm.

\medskip

The following Lemma allows us to compute the deficiency of the realization obtained from Algorithm~\ref{algorithm 1}.

\begin{lemma}
Suppose that the dynamical system $\dot{\bx} = \sum\limits_{i=1}^m \bx^{\by_i} \bw_i$ with $m$ vertices, and the matrix of net reaction vectors $\bW = (\bw_1, \ldots, \bw_m)$ passes Algorithm~\ref{algorithm 1} and outputs a weakly reversible realization consisting of a single linkage class. Then the deficiency of this realization is $m- 1 - \rm{Im}(\bW)$.
\end{lemma}

\begin{proof}

From Lemma~\ref{lem:wr_ker}, we have $\rm{Im}(W) =S$. Therefore, the deficiency of realization obtained from Algorithm~\ref{algorithm 1} is
\begin{equation} \notag
\delta = m - \ell - s= m- 1 - \rm{Im}(\bW).
\end{equation}
\end{proof}

\subsection{Weakly reversible deficiency zero realizations consisting of a single linkage class }

We first consider the case when a dynamical system admits a weakly reversible deficiency zero realization consisting of a single linkage class. 

It is well known that weakly reversible deficiency zero networks are complex-balanced for any choice of positive rate constants~\cite{horn1972general}. 
In addition, for complex-balanced dynamical systems consisting of a single linkage class, there exists a globally attracting positive steady state within each stoichiometric compatibility class~\cite{anderson2011proof}. This leads to the following Lemma. 

\begin{lemma} 
For a weakly reversible deficiency zero reaction network consisting of a single linkage class, every stoichiometric compatibility class admits a globally attracting positive steady state. 
\end{lemma}

This is our primary motivation for finding weakly reversible deficiency zero realizations consisting of a single linkage class. 
Now we state the upcoming Lemma that relates these realizations to the existence of a vector in line~\ref{algo_1_line:2} of Algorithm~\ref{algorithm 1}.

\begin{lemma}
\label{lem:single_linkage_kernel 0}

Consider a weakly reversible deficiency zero reaction network $G$ consisting of a single linkage class $L = \{\by_1, \ldots, \by_m \}$. Let $\{\bw_1, \ldots, \bw_m \}$ denote the net reaction vectors corresponding to these vertices. Define the matrix $\bB_i \in \mathbb{R}_{n \times m}$, with $k^{\rm th}$ column $B_{i, k} := (\by_k -\by_i)$ for $1 \leq k \leq m$. For each vertex $\by_i \in L$, there exists a unique vector $\bv\in \mathbb{R}^{m}_{\geq 0}$, such that 
\begin{equation} \label{alg 1 in 0}
\bB_i \bv = \bw_i, \ \rm{and } \ \bv_i = 1.
\end{equation}
\end{lemma}

\begin{proof}

For each vertex $\by_i \in L$, we write the stoichiometric subspace as 
\begin{equation} \notag
S = \text{span} \{\by_k - \by_i\}^{m}_{k=1}.
\end{equation}
Since $G$ has a single linkage class and deficiency zero, we obtain
\begin{equation} \notag
    0 = m - 1 - \dim (S).
\end{equation}
This shows that 
\begin{equation} \notag
\dim (\text{span} \{\by_k - \by_i\}^{m}_{k=1}) = \dim (S) = m - 1.
\end{equation}
From weak reversibility and Lemma~\ref{lem:wr_ker}, we deduce that  
\begin{equation} \notag
\dim (\ker (\bB_i)) = m - \dim(\rm{Im}(\bB_i)) = m - \dim (S) = 1.
\end{equation}
Since $\bB_{i, i} = \mathbf{0}$, it is easy to see that 
$\mathbf{e_i} \in \ker (\bB_i)$ where $\mathbf{e_i}$ represents the unit vector in the $i$-th coordinate. Then, we have 
\begin{equation} \label{dim kernel Bi 0}
\ker (\bB_i) \cap \{\bz\in\mathbb{R}^m: \bz_i = 0\} = \mathbf{0}.
\end{equation}
Since the net reaction vectors $\{\bw_i\}^m_{i=1}$ come from the dynamics generated by network $G$, all of them can be realized.
Applying~\eqref{dim kernel Bi 0}, we conclude that Equation~\eqref{alg 1 in 0} has a unique solution for each vertex $y_i \in L$.
\end{proof}

\begin{remark}
It is worth mentioning that if the dynamical system $\dot{\bx} = \sum_{i=1}^m \bx^{\by_i} \bw_i$ admits a weakly reversible realization consisting of a single linkage class $L$, such that for each vertex $y_i \in L$, the Equation~\ref{algo_1_line:2} in Algorithm~\ref{algorithm 1} has a unique solution, such realization still can have a positive deficiency. For example, the network in Example~\ref{ex: alg pass with rmk} has a unique solution to Equation~\ref{algo_1_line:2} for each vertex $y_i \in L$, but it has deficiency one.
\end{remark}

\begin{example}
\label{ex: alg pass def 0}
Consider the matrices corresponding to the source vertices and net reaction vectors given by
\begin{equation}
\bY_s =  \begin{pmatrix}
1 & 2 & 2 \\
0 & 0 & 1
\end{pmatrix}, \ \text{and } \
\bW =  \begin{pmatrix}
1 & 0 & -1 \\
0 & 1 & -1
\end{pmatrix}.
\end{equation}
respectively, which are inputs to Algorithm \ref{algorithm 1}. These inputs generate the following system of differential equations
\begin{equation} \label{eq: alg pass def 0}
\begin{split}
\dot{x} & = x - x^2y, \\
\dot{y} & = x^2 - x^2y.
\end{split}
\end{equation}
We have $n = 2$ for two state variables $x, y$, and $m = 3$ for two distinct monomials.

Next, applying line \ref{algo_1_line:2} in algorithm on $\bY_s = (\by_1, \by_2, \by_3)$, we obtain 
\begin{equation} \notag
\begin{split}
\bB_1 =  \begin{pmatrix}
0 & 1 & 1 \\
0 & 0 & 1
\end{pmatrix},& \ \ 
\bB_2 =  \begin{pmatrix}
-1 & 0 & 0 \\
 0 & 0 & 1
\end{pmatrix}, \ \
\bB_3 =  \begin{pmatrix}
-1 &  0 & 0 \\
-1 & -1 & 0
\end{pmatrix},
\end{split}
\end{equation}
and
\begin{equation} \notag
\bv^*_1 = (1,1,0)^{T}, \ \
\bv^*_2 = (0,1,1)^{T}, \ \
\bv^*_3 = (1,0,1)^{T},
\end{equation}
where $\bv^*_i \in \mathbb{R}^{3}_{\geq 0}$ and $\bB_i \bv^*_i = \bw_i$, for $i = 1, 2, 3$.

Then, we can compute that for $i = 1, 2, 3$,
\begin{equation}\label{eq:example_1_kernel}
\ker (\bB_i) \cap \{\bz\in\mathbb{R}^m: \bz_i = 0\} = \mathbf{0},
\end{equation}
and derive
\begin{equation} \notag
S_1 = \{ 1,2 \}, \ \
S_2 = \{ 2,3 \}, \ \
S_3 = \{ 1,3 \}.
\end{equation}
Note that $\rm{dim(ker}(B_i))=1$. Together with Equation~\eqref{eq:example_1_kernel}, we deduce that $\bv^*_i$ is the unique solution to the equations $\bB_i \bv = \bw_i$ and $\bv_i = 1$ for $i = 1, 2, 3$. Therefore, we do not need to execute the inner loop given by lines \ref{algo_1_line:9}-\ref{algo_1_line:17} in Algorithm~\ref{algorithm 1}. 

Following line~\ref{algo_1_line:7}, we construct the Kirchoff matrix:
\begin{equation} \notag
Q =  \begin{pmatrix}
-1 & 0 & 1 \\
1 & -1 & 0 \\
0 & 1 & -1 \\
\end{pmatrix}.
\end{equation}
It is easy to check that $\ker(Q) = \text{span} \{(1,1,1)^{\intercal}\}$, and we deduce that
\begin{equation} \notag
\dim (\ker(Q)) = 1, \ \text{and } \
\text{supp} (\ker(Q)) = \{ 1,2,3 \}.
\end{equation}
Therefore, we conclude that the system given by \eqref{eq: alg pass def 0} admits a weakly reversible realization with a single linkage class, whose E-graph is shown in Figure~\ref{fig:alg pass def 0}.

\begin{figure}[H]
\centering
\includegraphics[scale=0.5]{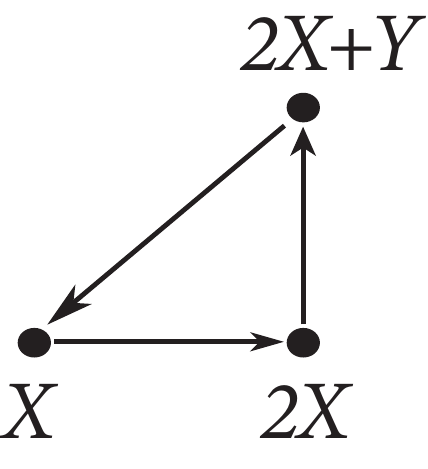}
\caption{The deficiency zero  mass-action system from Example \ref{ex: alg pass def 0}}
\label{fig:alg pass def 0}
\end{figure}

\end{example}

\begin{example}
Consider the matrices of source vertices and net reaction vectors given by
\begin{equation}
\bY_s =  \begin{pmatrix}
1 & 2 \\
0 & 0
\end{pmatrix}, \ \text{and } \
\bW = \begin{pmatrix}
-1 & 1 \\
0 &  0
\end{pmatrix}.
\end{equation}
respectively, which are inputs to Algorithm \ref{algorithm 1}. These inputs generate the following system of differential equations
\begin{equation} \label{ex: alg fail}
\begin{split}
\dot{x} & = - x +x^2, \\
\dot{y} & = 0.
\end{split}
\end{equation}
We have $n = 2$ for two state variables $x, y$, and $m = 2$ for two distinct monomials. 

Next, following line \ref{algo_1_line:2} in algorithm on $\bY_s = (\by_1, \by_2)$, we obtain 
\begin{equation}
\bB_1 =  \begin{pmatrix}
0 & 1 \\
0 & 0
\end{pmatrix}, \ \text{and } \ 
\bB_2 =  \begin{pmatrix}
-1 & 0 \\
0 & 0
\end{pmatrix}.
\end{equation}
However, there does not exist a positive vector $\bv^*$, which solves $\bB_1 \bv^* = \bw_1$. Therefore, there exists no weakly reversible realization consisting of a single linkage class that generates the dynamical system given by Equation~\ref{ex: alg fail}.
\end{example}

\subsection{Weakly reversible deficiency one realizations consisting of a single linkage class}

In this section, we analyze the case when a dynamical system admits a weakly reversible deficiency one realization consisting of a single linkage class. 

If a reaction network satisfies the conditions of the Deficiency One Theorem, then every stoichiometric compatibility class contains a unique positive steady state (if it exists)~\cite{feinberg1995existence,feinberg2019foundations}.
On the other hand, for any weakly reversible network, there always exists a positive steady state within every stoichiometric compatibility class~\cite{boros2019existence}.  
It is easy to check that every weakly reversible deficiency one network with a single linkage class must satisfy all conditions in the Deficiency One Theorem. Therefore, we get the following Lemma.

\begin{lemma} 
For a weakly reversible deficiency one network consisting of a single linkage class, there exists a unique positive steady state within every stoichiometric compatibility class.
\end{lemma}

This explains the importance of discovering weakly reversible deficiency one realizations with a single linkage class.
Moreover, we introduce the next Lemma showing the existence of a vector in line~\ref{algo_1_line:2} of Algorithm~\ref{algorithm 1}.

\begin{lemma}
\label{lem:single_linkage_kernel 1}

Consider a weakly reversible and deficiency one reaction network $G$ consisting of a single linkage class given by $L = \{\by_1, \ldots, \by_m \}$. Let $\{\bw_1, \ldots, \bw_m \}$ denote the net reaction vectors corresponding to these vertices. Define the matrix $\bB_i \in \mathbb{R}_{n \times m}$, with $k^{\rm th}$ column $B_{i, k} := (\by_k -\by_i)$ for $1 \leq k \leq m$. For each vertex $\by_i \in L$, the following system
\begin{equation} \label{alg 1 in 1}
\begin{split}
\bB_i \bv & = \bw_i, \\ 
\ \bv_i = 1,  \ & \text{and }  \
\bv \in \mathbb{R}^{m}_{\geq 0}.
\end{split}
\end{equation} 
has at most two linearly independent solutions.
\end{lemma}

\begin{proof}

For each vertex $y_i \in L$, we denote the stoichiometric subspace by $S$, such that
\begin{equation}
S = \text{span} \{\by_k - \by_i\}^{m}_{k=1}.
\end{equation}
Since $G$ has a single linkage class and deficiency one, we obtain
\begin{equation}
    \delta = 1 = m - 1 - \dim (S).
\end{equation}
This shows that 
\begin{equation}
    \dim (\text{span} \{\by_k - \by_i\}^{m}_{k=1}) = \dim (S) = m - 2.
\end{equation}
From the Rank-Nullity Theorem, $\dim (\ker (\bB_i)) + \dim(\rm{Im}(\bB_i)) = m$. Since $G$ is weakly reversible, using Lemma~\ref{lem:wr_ker}, we get 
\begin{equation}
\dim(\rm{Im}(\bB_i)) = \dim (S),
\end{equation}
thus we obtain that $\dim (\ker (\bB_i)) = m - \dim (S) = 2$.

Since $\bB_{i, i} = \mathbf{0}$, we deduce $\ker (\bB_i)$ has one vector $\bu\in\mathbb{R}^m$ such that $\bu_i \neq 0$.
Then, we have 
\begin{equation} \label{dim kernel Bi 1}
\dim (\ker (\bB_i) \cap \{\bz\in\mathbb{R}^m: \bz_i = 0\}) = 1.
\end{equation}
Since the net reaction vectors $\{\bw_i\}^m_{i=1}$ come from the dynamics generated by the network $G$, all of them can be realized. Together with~\eqref{dim kernel Bi 1}, the conclusion follows.
\end{proof}

\begin{example}
\label{ex: alg pass with rmk}

Consider the matrices corresponding to the source vertices and net reaction vectors given by
\begin{equation}
\bY_s =  \begin{pmatrix}
1 & 2 & 2 & 1 \\
0 & 0 & 1 & 1
\end{pmatrix}, \ \text{and } \
\bW =  \begin{pmatrix}
1  & -1  & 0  & 0\\
0 	& 0 & 1  & -1
\end{pmatrix}.
\end{equation}
respectively, which are inputs to Algorithm \ref{algorithm 1}. These inputs generate the following system of differential equations
\begin{equation} \label{eq: alg pass with rmk}
\begin{split}
\dot{x} & = x - x^2y, \\
\dot{y} & = x^2 - xy.
\end{split}
\end{equation}
We have $n = 2$ for two state variables $x, y$, and $m = 4$ for four distinct monomials. 

Next, applying line \ref{algo_1_line:2} in algorithm on $\bY_s = (\by_1, \by_2, \by_3, \by_4)$, we obtain 
\begin{equation} \notag
\begin{split}
\bB_1 =  \begin{pmatrix}
0 & 1 & 1 & 0 \\
0 & 0 & 1 & 1
\end{pmatrix},& \ \ 
\bB_2 =  \begin{pmatrix}
-1 & 0 & 0 & -1 \\
 0 & 0 & 1 & 1
\end{pmatrix}, \\
\bB_3 =  \begin{pmatrix}
-1 & 0 & 0 & -1 \\
-1 & -1 & 0 & 0
\end{pmatrix},& \ \ 
\bB_4 =  \begin{pmatrix}
0 & 1 & 1 & 0 \\
-1 & -1 & 0 & 0
\end{pmatrix},
\end{split}
\end{equation}
and
\begin{equation} \notag
\bv^*_1 = (1,1,0,0)^{T}, \ \bv^*_2 = (0,1,1,0)^{T}, \
\bv^*_3 = (0,0,1,1)^{T}, \ \bv^*_4 = (1,0,0,1)^{T},
\end{equation}
where $\bv^*_i \in \mathbb{R}^{4}_{\geq 0}$ and   $\bB_i \bv^*_i = \bw_i$, for $1 \leq i \leq 4$.

Then, we get the initial $S_i$ for $1 \leq i \leq 4$,
\begin{equation} \notag
S_1 = \{ 1,2 \}, \ \
S_2 = \{ 2,3 \}, \ \
S_3 = \{ 3,4 \}, \ \
S_4 = \{ 1,4 \}.
\end{equation}
After executing the inner loop in lines \ref{algo_1_line:9}-\ref{algo_1_line:17}, we do not have any update on $S_i$.

Now we follow line~\ref{algo_1_line:7}, and construct the Kirchoff matrix:
\begin{equation} \notag
Q =  \begin{pmatrix}
-1 & 0 & 0 & 1 \\
1 & -1 & 0 & 0 \\
0 & 1 & -1 & 0 \\
0 & 0 & 1 & -1
\end{pmatrix}.
\end{equation}
It is easy to check that $\ker(Q) = \text{span} \{(1,1,1,1)^{\intercal}\}$, which shows
\begin{equation} \notag
\dim (\ker(Q)) = 1, \ \text{and } \
\text{supp} (\ker(Q)) = \{ 1,2,3,4 \}.
\end{equation}
Therefore, we conclude that~\eqref{eq: alg pass with rmk} admits a weakly reversible realization with a single linkage class, whose E-graph is shown in Figure \ref{fig:alg pass with rmk}.

\begin{figure}[H]
\centering
\includegraphics[scale=0.5]{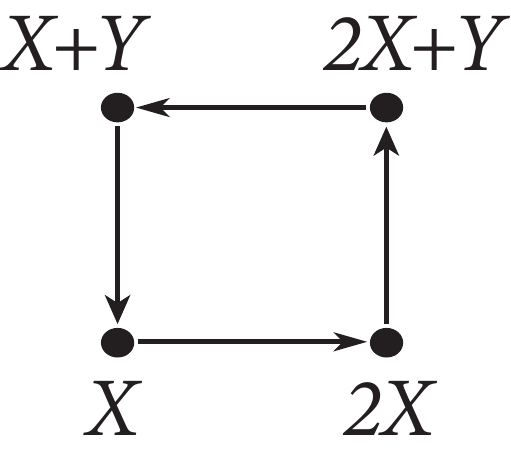}
\caption{The deficiency one  mass-action system from Example \ref{ex: alg pass with rmk}}
\label{fig:alg pass with rmk}
\end{figure}

\end{example}

\begin{example}
\label{ex: alg pass}
Consider the matrices corresponding to the source vertices and net reaction vectors given by
\begin{equation}
\bY_s =  \begin{pmatrix}
1 & 2 & 3
\end{pmatrix}, \ \text{and } \
\bW =  \begin{pmatrix}
1  & 1 & -1
\end{pmatrix}.
\end{equation}
respectively, which are inputs to Algorithm \ref{algorithm 1}. These inputs generate the following differential equation
\begin{equation} \label{eq: alg pass}
\dot{x} = x + x^2 - x^3.
\end{equation}
We have $n = 1$ for the state variables $x$, and $m = 3$ for three distinct monomials. 

Next, applying line \ref{algo_1_line:2} in algorithm on $\bY_s = (\by_1, \by_2, \by_3)$, we obtain 
\begin{equation} \notag
\bB_1 =  \begin{pmatrix}
0 & 1 & 2
\end{pmatrix}, \ \ 
\bB_2 =  \begin{pmatrix}
-1 & 0 & 1
\end{pmatrix}, \ \
\bB_3 =  \begin{pmatrix}
-2 & -1 & 0
\end{pmatrix},
\end{equation}
and
\begin{equation} \notag
\bv^*_1 = (1,1,0)^{T}, \ \
\bv^*_2 = (0,1,1)^{T}, \ \
\bv^*_3 = (0,1,1)^{T},
\end{equation}
where $\bv^*_i \in \mathbb{R}^{3}_{\geq 0}$ and $\bB_i \bv^*_i = \bw_i$, for $i = 1, 2, 3$.

Then, we get the initial $S_i$ for $i = 1, 2, 3$,
\begin{equation} \notag
S_1 = \{ 1,2 \}, \ \
S_2 = \{ 2,3 \}, \ \
S_3 = \{ 2,3 \}.
\end{equation}
Following the inner loop in lines \ref{algo_1_line:9}-\ref{algo_1_line:17}, we can compute that 
\begin{equation} \notag
\begin{split}
& \bB_1 \bv^1 = \bw_1, \ \text{with } \ \bv^1 = (1,0,1/2)^{\intercal},
\\& \bB_2 \bv^2 = \bw_2, \ \text{with } \ \bv^2 = (1,1,2)^{\intercal},
\\& \bB_3 \bv^3 = \bw_3, \ \text{with } \ \bv^3 = (1/2,0,1)^{\intercal}.
\end{split}
\end{equation}
After updating $S_i$ with $\bv^i$ for $i = 1, 2, 3$, we derive 
\begin{equation} \notag
S_1 = S_2 = S_3 = \{ 1,2,3 \}.
\end{equation}

Now we follow line \ref{algo_1_line:7}, and construct the Kirchoff matrix:
\begin{equation} \notag
Q =  \begin{pmatrix}
-2 & 1 & 1 \\
1 & -2 & 1 \\
1 & 1 & -2 \\
\end{pmatrix}.
\end{equation}
It is easy to check that $\ker(Q) = \text{span} \{(1,1,1)^{\intercal}\}$, and we deduce that
\begin{equation} \notag
\dim (\ker(Q)) = 1, \ \text{and } \
\text{supp} (\ker(Q)) = \{ 1,2,3 \}.
\end{equation}
Therefore, we conclude \eqref{eq: alg pass} admits a weakly reversible realization with a single linkage class, whose E-graph is shown in Figure \ref{fig:alg pass}.

\begin{figure}[H]
\centering
\includegraphics[scale=0.5]{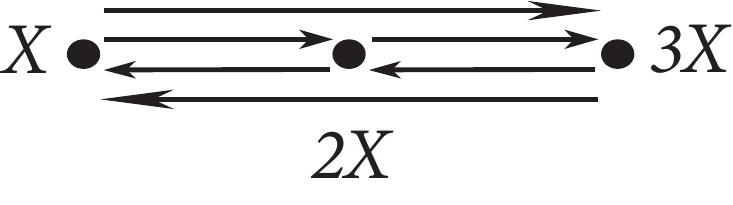}
\caption{The deficiency one mass-action system from Example \ref{ex: alg pass}}
\label{fig:alg pass}
\end{figure}

\end{example}

\subsection{Weakly reversible realizations with a single linkage class have high deficiency}

Now we list some properties of weakly reversible realizations of arbitrary positive deficiency consisting of a single linkage class. 

Our motivation comes from autocatalytic networks, which are often associated with the context of the origin of life models~\cite{deshpande2014autocatalysis,hordijk2010autocatalytic,hordijk2004detecting,hordijk2012structure}. Owing to their autocatalytic nature, the concentrations of species in these networks can go unbounded. The crucial component in their analysis is the dynamics corresponding to the relative concentration of species. Given species $X_1, X_2,\ldots, X_n$ with concentrations $x_1, x_2, \ldots, x_n$, the relative concentration corresponding to species $X_i$ is given by $(\sum_{i=1}^n x_i)^{-1} x_i$. It can be shown that for certain autocatalytic networks, the dynamics corresponding to the relative concentration of species can be generated by a reaction network~\cite{craciun2022autocatalytic}. 

In particular, we present an example of an autocatalytic network such that the network corresponding to the relative concentration of species is weakly reversible and consists of a single linkage class. Table~\ref{table:autocatalytic} illustrates this fact. The left column of the table describes the reactions in $G$, which is an autocatalytic network. The right column of the table describes the reactions in $\tilde{G}$, which is the network corresponding to the relative concentration of species in $G$. The reactions in $\tilde{G}$ are obtained in the following way: for every reaction in $G$, there exists a corresponding pair of reactions in $\tilde{G}$ that is generated using  ~\cite[Theorem 3.5]{craciun2022autocatalytic}. In particular, the reactions in $\tilde{G}$ are generated by adding all possible species to the reactants of the corresponding reaction in $G$.

\begin{table}[H]
\caption{}
\centering
\begin{tabular}{|c|c|}
\hline
\rule{0pt}{20pt} Reactions in $G$ & Reactions in $\tilde{G}$  
\\ [2ex]
\hline
\begin{tabular}{l}
$X_1 + X_2 \xrightarrow[]{k_1} 2X_1 + X_2$
\end{tabular} &
\begin{tabular}{l}
$X_1 + 2X_2 \xrightarrow{k_1} 2X_1 + X_2$ \\
$X_1 + X_2 + X_3 \xrightarrow{k_1} 2X_1 + X_2$
\end{tabular} 
\\
\hline 
\begin{tabular}{l}
$X_2 + X_3 \xrightarrow[]{k_2} 2X_2 + X_3$
\end{tabular} &
\begin{tabular}{l}
$X_2 + 2X_3 \xrightarrow{k_2} 2X_2 + X_3$ \\
$X_1 + X_2 + X_3 \xrightarrow{k_2} 2X_2 + X_3$
\end{tabular} 
\\
\hline
\begin{tabular}{l}
$X_3 + X_1 \xrightarrow[]{k_3} 2X_3 + X_1$ 
\end{tabular} & 
\begin{tabular}{l}
$2X_1 + X_3 \xrightarrow{k_3} 2X_3 + X_1$ \\
$X_1 + X_2 + X_3 \xrightarrow{k_3}2X_3 + X_1$ 
\end{tabular}
\\
\hline
\begin{tabular}{l}
$X_1 + X_2 \xrightarrow[]{k_4} X_1 + X_2 + X_3$ 
\end{tabular} & 
\begin{tabular}{l}
$X_1 + 2X_2 \xrightarrow{k_4} X_1 + X_2 +X_3$ \\
$2X_1 + X_2 \xrightarrow{k_4} X_1 + X_2 +X_3$
\end{tabular}
\\
\hline
\begin{tabular}{l}
$X_2 + X_3 \xrightarrow[]{k_5} X_1 + X_2 + X_3$
\end{tabular} & 
\begin{tabular}{l}
$X_2 + 2X_3 \xrightarrow{k_5} X_1 + X_2 +X_3$ \\
$2X_2 + X_3 \xrightarrow{k_5} X_1 + X_2 +X_3$ 
\end{tabular}
\\
\hline
\begin{tabular}{l}
$X_1 + X_3 \xrightarrow[]{k_6} X_1 + X_2 + X_3$ 
\end{tabular} & 
\begin{tabular}{l}
$X_1 + 2X_3 \xrightarrow{k_6} X_1 + X_3 +X_3$ \\ 
$2X_1 + X_3 \xrightarrow{k_6} X_1 + X_2 +X_3$ 
\end{tabular}
\\
\hline
\end{tabular}
\label{table:autocatalytic}
\end{table}

\medskip

The network $\tilde{G}$ is depicted in Figure~\ref{fig:non_lotka_volterra}.(a).  Note that the deficiency of $\tilde{G}$ is given by $\delta = 7 - 1 - 3 = 3$.  Using some modifications, we can construct a network shown in Figure~\ref{fig:non_lotka_volterra}.(b) which generates the dynamics as Figure~\ref{fig:non_lotka_volterra}.(a).  Figure~\ref{fig:non_lotka_volterra}.(b) is a weakly reversible network consisting of a single linkage class. By~\cite{boros2020permanence}, the dynamics generated by it is permanent. This implies that the dynamics generated by $\tilde{G}$ is also permanent. 

\begin{figure*}[ht]
\centering
\includegraphics[scale=0.4]{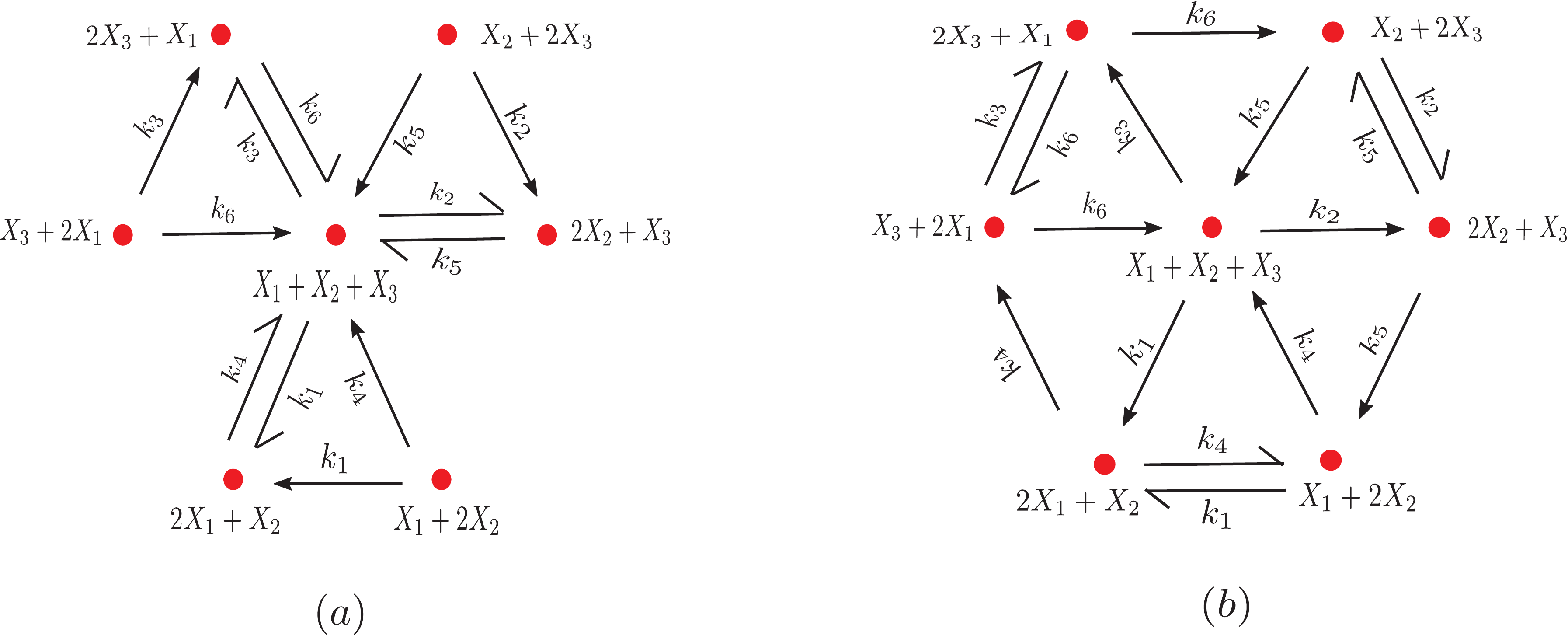}
\caption{(a) The network $\tilde{G}$ corresponds to the relative concentrations of species in network $G$. (b) Splitting certain reactions in (a) gives a weakly reversible network consisting of a single linkage class. The dynamics generated by this network is known to be permanent~\cite{boros2020permanence}.}
\label{fig:non_lotka_volterra}
\end{figure*} 

\subsection{Implementation of Algorithm~\ref{algorithm 1}}
\label{sec:complexity}

In this section, we discuss the implementation aspects of Algorithm~\ref{algorithm 1}. The algorithm is designed to find a weakly reversible realization consisting of a single linkage class for $\dot{\bx} = \sum\limits_{i=1}^m \bx^{\by_i} \bw_i$, and it has three key steps:

\begin{enumerate}
\item Check for the existence of a vector $\bv^* \in \mathbb{R}^{m}_{\geq 0}$, such that for $i = 1, \ldots, m$,
\begin{equation} \notag
\bB_i \bv^* = \bw_i.
\end{equation}

\item Check for the existence of a vector $\bv \in \mathbb{R}^{m}_{\geq 0}$, such that for $i, j = 1, \ldots, m$,
\begin{equation} \notag
\bB_i \bv = \bw_i, \ \text{and } \
\bv_j > 0.
\end{equation} 

\item Check $\dim (\ker(Q)) = 1$, and $\text{supp} (\ker(Q)) = \{1, \ldots, m \}$. \smallskip
\end{enumerate}

\medskip

In step 1, we compute the positive vector solving $\bB_i \bv^* = \bw_i$ and consider the implementation as a sequence of linear programming problems. For $i = 1, \ldots, m$, set the matrix $\bB_i \in \mathbb{R}_{n \times m}$ as in line \ref{alg: Bi},
\begin{equation} \label{step 1 implement}
\begin{array}{llr} 
    \text{Find a vector} & \bx,  \\ 
    \text{subject to}  &  \bB_i \bx = \bw_i, \\ 
      & \bx \geq \mathbf{0}.
\end{array} 
\end{equation}
From Lemma~\ref{lem: algorithm 1 part b}, if there exists some number $1 \leq i \leq m$, such that there is no solution for \eqref{step 1 implement}, then the implementation fails.
Therefore, no weakly reversible realization with a single linkage class exists. 

\medskip

In step 2,  we compute the positive vector solving $\bB_i \bv = \bw_i$ and $\bv_j > 0$. For each $j = 1, \ldots, m$, if $j \notin S_i$, we do the following:
\begin{equation} \label{step 2 implement}
\begin{array}{llr} 
    \text{Find a vector} & \bx,  \\ 
    \text{that maximizes} & \bx_j,  \\ 
    \text{subjected to}  &  \bB_i \bx = \mathbf{0}, \\ 
     & \bx \geq \mathbf{0}, \ \text{and } \ 
     \bx_j \leq 1.
\end{array} 
\end{equation}
It is clear that the solution to~\eqref{step 2 implement} is the desired vector if its $j$-th component is positive. Meanwhile, if the $j$-th component of the solution is zero, then implementation fails. Furthermore, we restrict $\bx_j \leq 1$ to avoid the risk that $\bx_j$ can be arbitrarily large.

Here we explain why adding the restriction on $j$-th component in \eqref{step 2 implement} does not change the solvability of the problem. From $j \notin S_i$, there must exist a vector $\bx^* \in \mathbb{R}^{m}_{\geq 0}$, such that
\begin{equation*}
\bB_i \bx^* = \bw_i, \ \text{and } \ 
\bx^*_j = 0.
\end{equation*}
Suppose there is a vector $\bx \in \mathbb{R}^{m}_{\geq 0}$, which solves
\begin{equation} \label{step 2 implement 2}
\bB_i \bx = \bw_i, \ \text{and } \ 
\bx_j > 0.
\end{equation}
Then we can always find a sufficient small constant $\epsilon$ with $\bx^{\epsilon} := (1-\epsilon )\bx^* + \epsilon \bx$, such that 
\begin{equation} \notag
\bB_i \bx^{\epsilon} = \bw_i, \ \text{and } \
0 < \bx^{\epsilon}_{j} = \epsilon \bx_j \leq 1.
\end{equation}
This implies that if \eqref{step 2 implement 2} admits a solution, there must exist another solution for~\eqref{step 2 implement}.

\medskip

In step 3, the implementation needs a rank-revealing factorization; we need to find a basis of $\ker(Q)$, and then we can check the number of vectors in this basis and their support. This is again solving a linear programming problem.

\section{Discussion}\label{sec:discussion}

Weakly reversible networks consisting of a single linkage class form an important class of networks, owing to the robust properties of the dynamical systems they generate. 
In particular, the dynamics produced by these networks (according to mass-action kinetics) is known to be persistent and permanent for all choices of reaction rate parameters~\cite{gopalkrishnan2014geometric,boros2020permanence}. 

We describe an algorithm that determines if there exists a weakly reversible realization consisting of a single linkage class that generates a given polynomial dynamical system. Our input consists of two matrices: a matrix of source monomials and a matrix containing the corresponding net reaction vectors. The algorithm outputs a {maximal} weakly reversible realization consisting of a single linkage class (if one exists), which generates the dynamical system formed by the inputs. We also describe  approaches for efficient  implementations of this algorithm; in particular, we show that all the key steps in our algorithm reduce to solving simple linear programming problems.

Other approaches for finding weakly reversible realizations of polynomial dynamical systems are based primarily on mixed integer programming methods~\cite{johnston2013computing,szederkenyi2013optimization,rudan2014efficiently,szederkenyiweak2011finding}. The algorithm we describe here uses a simpler \emph{greedy} approach, which works specifically because  we are looking for realizations consisting of a {\em single} linkage class.  

At the same time, our algorithm lays down the foundation for some future work. In particular, extending our algorithm to check the existence of more general realizations (e.g., weakly reversible realizations with {\em multiple} linkage classes that satisfy other desirable properties) is a potential avenue worthy of exploration. More specifically, the problem of finding weakly reversible realizations that satisfy the conditions of the {\em Deficiency One Theorem}~\cite{feinberg2019foundations} is a possibility that will explore in a follow-up paper~\cite{WR_DEF_THM}.

\section*{Acknowledgements} 

This work was supported in part by the National Science Foundation  grant DMS-2051568.

\section*{Data availability}

Data sharing not applicable to this article as no datasets were generated or analyzed during the current study.

\bibliographystyle{unsrt}
\bibliography{Bibliography}

\end{document}